\documentclass[11pt, a4paper]{article}

\usepackage[T1]{fontenc}
\usepackage[utf8]{inputenc}
\usepackage{xcolor} %used for font color
\usepackage{graphicx} % for pdf, bitmapped graphics files
\usepackage{times} % assumes new font selection scheme installed
\usepackage{amsmath} % assumes amsmath package installed
\usepackage{amssymb}  % assumes amsmath package installed
\usepackage{subcaption}

\usepackage{url}
\usepackage{multirow}
\usepackage{cite}
\usepackage{algorithm}

\usepackage{algorithmic}
\usepackage{microtype} % improves justification, reduces overfull hboxes
\newcommand{\qedsymbol}{\hfill$\blacksquare$}

\graphicspath{{./figures/}}

\newtheorem{theorem}{Theorem}[section]
\newtheorem{lemma}[theorem]{Lemma}

\newtheorem{proposition}[theorem]{Proposition}
\newtheorem{st_assumption}[theorem]{Standing Assumption}
\newtheorem{definition}[theorem]{Definition}

\newtheorem{remark}[theorem]{Remark}

\newenvironment{proof}{\emph{Proof.}}{\hfill\qedsymbol}

\newcommand{\rr}{\mathbb{R}}

\newcommand{\tx}{\tilde x}
\newcommand{\tl}{\tilde \lambda}
\newcommand{\tn}{\tilde \nu}
\newcommand{\xstar}{x^\star}
\newcommand{\lstar}{\lambda^\star}
\newcommand{\nstar}{\nu^\star}
\newcommand{\T}{\top}

\newcommand{\XX}{\mathcal{X}}
\newcommand{\II}{\mathcal{I}}
\newcommand{\NN}{\mathcal{N}}

\newcommand{\cvd}{\hfill\qedsymbol}

\begin{document}

\title{Solving Monotone Linear-Quadratic Generalized Nash Equilibrium Problems via Quadratic Programming}

\author{Alberto Bemporad and Tatiana Tatarenko 
{\renewcommand{\thefootnote}{}\thanks{A. Bemporad is with the IMT School for Advanced Studies, Piazza San Francesco 19, Lucca, Italy. Email: \texttt{\scriptsize alberto.bemporad@imtlucca.it}. T. Tatarenko is with the Department of Control Theory and Intelligent Systems, TU Darmstadt, Germany. E-mails: \texttt{\scriptsize tatiana.tatarenko@tu-darmstadt.de}.   
This work was funded by the European Union (ERC Advanced Research Grant COMPACT, No. 101141351). Views and opinions expressed are however those of the authors only and do not necessarily reflect those of the European Union or the European Research Council. Neither the European Union nor the granting authority can be held responsible for them. The work was also funded by the Deutsche Forschungsgemeinschaft (DFG, German Research Foundation). Project number: 528033031.
}}}

\maketitle
\thispagestyle{empty}
	
\begin{abstract}
We consider generalized Nash equilibrium problems among $N$ players with convex quadratic costs and shared affine constraints, assuming only that the game's pseudogradient is merely monotone. We show that computing a variational generalized Nash equilibrium (v-GNE) is equivalent to solving a single convex quadratic program (QP) derived from the players' joint Karush--Kuhn--Tucker conditions. Building on this,
we show that the regularization of such a QP yields an $\varepsilon$-approximated v-GNE
with suboptimality vanishing linearly in the regularization parameter. Next, we propose an accelerated proximal-point scheme and an accelerated projected-gradient method, both attaining an $\mathcal O(1/k^2)$-approximated v-GNE at the $k$-th iteration. We also demonstrate that an invertible Jacobian of the game allows for reduction to a lower-dimensional QP. Theoretical analysis and numerical experiments show the proposed methods substantially outperform the existing approaches to solve monotone linear-quadratic v-GNE problems.
\end{abstract}

% \noindent{\small \textbf{Keywords:} Generalized Nash equilibrium problems, game-theoretic optimization, quadratic programming, game-theoretic model predictive control, monotone games, variational inequalities.}

\section{Introduction}

Generalized Nash equilibrium problems (GNEPs) have become a fundamental framework for modeling decentralized decision-making and distributed control in engineering systems, where multiple strategic agents optimize individual objectives subject to shared coupling constraints. In recent years, they have attracted considerable attention in systems and control, with applications including power and energy systems, communication networks, autonomous vehicles, and resource allocation. In particular, game-theoretic model predictive control (MPC) has emerged as a powerful paradigm for coordinating dynamically coupled systems while preserving the autonomy of individual decision makers. In these formulations, each sampling instant requires the solution of a GNEP, making the overall performance of the controller critically dependent on the availability of efficient, reliable, and scalable equilibrium solvers~\cite{HB26,LSM22}.

Many game-theoretic formulations in control theory, including game-theoretic model predictive control, naturally lead to linear-quadratic generalized Nash equilibrium problems (LQ-GNEPs), where the agents' objectives are quadratic and the coupling constraints are affine~\cite{Bem26}. It is worth noting that solving LQ-GNEPs remains considerably more challenging than solving centralized quadratic optimization problems because each player optimizes an individual objective subject to a constraint set depending on the strategies of the other players. However, given the monotonicity of the game's pseudogradient, one can utilize the structure of the problem to develop efficient solution methods. Such methods typically formulate the problem as a variational inequality and employ forward-backward, proximal-point, Tikhonov regularization or extragradient-type algorithms~\cite{FP03,BYGP22}. While these approaches are well established and supported by strong theoretical guarantees, many of them require strong monotonicity of the game's pseudogradient to guarantee last iterate convergence with estimations of convergence rates~\cite{TatKam25}. Unfortunately, strong monotonicity is often difficult to verify and may fail in practical game-theoretic MPC applications due to strong interactions among agents or the presence of multiple equilibria. This motivates the development of computational methods that remain robust under the substantially weaker assumption of \emph{mere monotonicity}.

From a broader perspective, equilibrium problems and optimization problems represent two fundamentally different computational paradigms. Convex optimization enjoys a mature theoretical foundation together with highly efficient numerical algorithms and software capable of solving large-scale problems. Consequently, reformulating equilibrium problems as equivalent optimization problems has long been recognized as an attractive research direction. Such reformulations not only provide valuable structural insight into the underlying game but also enable the direct use of the extensive algorithmic machinery developed for convex optimization.
For Nash and generalized Nash equilibrium problems, however, exact optimization reformulations are known only in rather special situations. The classical example is that of potential games \cite{MondererShapley1996}, where the players' objectives admit a common potential function whose minimizers coincide with the equilibria. Unfortunately, the existence of a potential function requires restrictive integrability conditions, namely that the pseudogradient is the gradient of a scalar function. These conditions are rarely satisfied in engineering applications, and most monotone games arising in practice are therefore not potential games.

Outside the class of potential games, optimization reformulations become considerably more difficult. 
 One may reformulate the equilibrium conditions using nonlinear complementarity systems or construct optimization problems based on gap or merit functions \cite{Fukushima1992,Kanzow1996,NesterovScrimali}. While these techniques provide powerful theoretical  tools, they generally do not yield tractable convex optimization problems equivalent to the original game. In particular, merit functions, although convex for the case of monotone games, lack smoothness, and are defined as the supremum of a family of convex functions. Consequently, they are generally nonsmooth and cannot be evaluated without solving an additional optimization problem. As a result, the reformulated optimization problem may remain as difficult as the original game-theoretic one. To the best of our knowledge, no exact tractable convex optimization reformulation is currently available for a broad class of non-potential monotone v-GNEPs, including the important case of linear-quadratic games with shared affine constraints.

The objective of this paper is to bridge this gap for LQ-GNEPs with shared affine constraints. We start from the joint Karush--Kuhn--Tucker conditions characterizing v-GNE. Our first main result establishes that these conditions are precisely the optimality conditions of a \emph{single convex quadratic program}, whose Hessian is given by the symmetric part of the game's pseudogradient. Consequently, every v-GNE is obtained as a solution of this quadratic program, and conversely.  This establishes an exact optimization reformulation for an important subclass of  generally non-potential monotone games.

The proposed reformulation has important computational consequences. It immediately enables the use of convex quadratic optimization,  including highly optimized sparse QP solvers, regularized active-set and interior-point methods, proximal-point techniques, warm-start strategies, and accelerated first-order optimization algorithms. 
As for specific theoretical results, we show that Tikhonov regularization of the obtained quadratic program produces an $\varepsilon$-generalized Nash equilibrium whose approximation error decreases linearly with the regularization parameter.
Moreover, building upon the proposed reformulation, we focus on two accelerated solution algorithms. The first one combines the proximal-point framework with accelerated solution of strongly convex quadratic subproblems, while the second applies Nesterov acceleration to projected-gradient iterations. We show that the $k$th iterate of each method is an $\mathcal{O}(1/k^2)$-approximated v-GNE, improving upon the $\mathcal{O}(1/\sqrt{k})$ approximation attained by standard extragradient methods. Additionally, for games with an invertible pseudogradient, we derive an equivalent lower-dimensional quadratic program involving only the dual variables, leading to substantial computational savings. Finally, the proposed framework is illustrated on randomly generated LQ-GNEPs and on a game-theoretic model predictive control problem. Both theoretical analysis
and numerical experiments demonstrate that the proposed approaches substantially outperform the extragradient method and other existing approaches while preserving the robustness associated with monotone game formulations.

\subsection{Notation}
We denote by $\rr^n$ the $n$-dimensional Euclidean space and by $\rr^{m\times n}$ the set of real $m\times n$ matrices. Given a matrix $A\in\rr^{m\times n}$, $A^\T$ denotes its transpose and $A_{i,:}$ denotes its $i$-th row. The set of indices $\{1, \ldots, N\}$ is denoted by $[N]$.

\section{Problem Formulation and Preliminaries}
Consider an $N$-player non-cooperative generalized Nash equilibrium problem with 
decision vectors $x_i \in \rr^{n_i}$, $i \in [N]$, and denote
by $x = (x_1, \ldots, x_N) \in \rr^n$, $n = \sum_{i=1}^N n_i$,
the joint strategy profile. The vector $x_{-i}$ denotes all the components of $x$ except those in $x_i$. In particular, without abuse of notation, we may write $x=(x_i, x_{-i})$. Each player $i$ minimizes its local convex quadratic cost function $J_i:\rr^n\to \rr$ by solving the following convex optimization problem:
\begin{equation}
    \begin{aligned}
    \min_{x_i} &\, \left\{J_i(x) =\tfrac{1}{2} x^\T Q^{(i)} x + (c^{(i)})^\T x\right\} \\
    \textrm{s.t. }& Ax \leq b,\ Ex = e
    \end{aligned}
    \label{eq:QPi}
\end{equation}
where $Q^{(i)}\in \rr^{n \times n}$, $c^{(i)}\in \rr^n$, $A \in \rr^{m \times n}$, $b \in \rr^m$, $E \in \rr^{q \times n}$, $e \in \rr^q$ ($m,q\geq 0$).
Without loss of generality, we assume $Q^{(i)}=(Q^{(i)})^\T$. 
We consider the constraint set 
\begin{equation}
\XX = \{x \in \rr^n : Ax \leq b,\; Ex = e\}
\label{eq:XX}
\end{equation}
as \emph{shared}, i.e., it is identical and common to all players.
Possible local inequality constraints only involving $x_i$, such as lower and upper bounds, as well as local equality constraints, are assumed to be incorporated into $Ax \leq b$ and $Ex = e$, respectively.

By stacking the gradients $\nabla_{x_i}J_i(x_i,x_{-i})=Q^{(i)}_{i, :}x + c^{(i)}_{i}$ of each agent's cost with respect to its own decision variable, we obtain the \emph{pseudogradient} $\phi: \rr^n \to \rr^n$ of the game 
\begin{equation}
\phi(x) = F x + f
\label{eq:pseudo-gradient}
\end{equation}
where $F \in \rr^{n \times n}$, $f \in \rr^n$, and
\begin{equation}\label{eq:Gg_assembly}
  F_{i, :} = Q^{(i)}_{i,:}, \qquad
  f_{i}   = c^{(i)}_{i}, \qquad i = 1,\ldots,N.
\end{equation}
In this work, we consider the following standing assumption regarding the game's structure: 
\begin{st_assumption}\label{assum:mon}
The pseudogradient $\phi$ of the game is a monotone mapping, i.e. 
\[(\phi(x) - \phi(y))^\T(x-y)\ge 0, \quad \forall x,y\in \rr^n.\]
\end{st_assumption}
The assumption above implies, in particular, that each cost function $J_i$ is convex in $x_i$, i.e., $Q^{(i)}_{ii} \succeq 0$, $\forall i \in [N]$,
where $Q^{(i)}_{ii}$ denotes the $i$-th diagonal block of $Q^{(i)}$. Thus, taking into account affinity of the shared constraints, we conclude that the game under consideration is a convex GNEP. We proceed with the definition of solution concept to such game.

\begin{definition} A vector $\xstar \in \XX$ is a \emph{generalized Nash equilibrium} (GNE) if each vector $\xstar_i$ solves~\eqref{eq:QPi} for $x_{-i} = \xstar_{-i}$, for all $i \in [N]$.
\end{definition}
Next, we provide the definition of $\varepsilon$-approximated GNE as follows. 
\begin{definition}\label{def:appGNE} A vector $x^{\varepsilon} \in \XX$ is an $\varepsilon$-\emph{approximated generalized Nash equilibrium} ($\varepsilon$-appGNE), if for every player $i \in [N]$ and every unilateral feasible deviation $z_i$ satisfying
$(z_i,x_{-i}^{\varepsilon}) \in \XX$, the following inequality holds:
$J_i(x_i^{\varepsilon},x_{-i}^{\varepsilon}) \le J_i(z_i,x_{-i}^{\varepsilon}) + \varepsilon$.
\end{definition}
According to the definition above, at any $\varepsilon$-appGNE, no player can decrease its cost by more than \(\varepsilon\)
through any feasible unilateral deviation.

It is worth noting that computing even an $\varepsilon$-appGNE is generally a challenging task~\cite{FK10}. Nevertheless, when the game is convex, as assumed throughout this paper (see the remark following Standing Assumption~\ref{assum:mon}), one can focus on a distinguished subset of GNEs, known as \emph{variational generalized Nash equilibria} (v-GNEs)~\cite{FK10}. Such equilibria are characterized as solutions of the variational inequality $VI(\XX,\phi)$:
\begin{equation}
\mbox{Find } x^\star \in \XX \mbox{ such that } \phi(x^\star)^\T (y - x^\star) \geq 0,
\quad \forall\, y \in \XX.
\label{eq:VI}
\end{equation}

Several methods exist to solve the variational inequality~\eqref{eq:VI} associated with~\eqref{eq:QPi}.
Most of the methods require {\it strong} monotonicity of the pseudogradient $\phi$, i.e., $(\phi(x) - \phi(y))^\T(x-y)\ge \mu\|x-y\|^2$ for some $\mu>0$, to guarantee convergence to a v-GNE.
As we only assume {\it mere} monotonicity of $\phi$ in Assumption~\ref{assum:mon} ($\mu=0$), to have a baseline for comparison with our proposed approach to solve merely-monotone linear-quadratic generalized Nash equilibrium problems (LQ-GNEP) as defined in~\eqref{eq:QPi}--\eqref{eq:XX}, we review below the well-established extragradient method~\cite{Kor76}.

\subsection{Extragradient Method}
\label{sec:extragradient}
Starting from any $x^0 \in \XX$, each iteration performs two projections onto $\XX$:
\begin{subequations}
\label{eq:extragradient}
\begin{align}
    y^k &= P_\XX(x^k - \alpha\,\phi(x^k)) \label{eq:extragradient_y}\\
    x^{k+1} &= P_\XX(x^k - \alpha\,\phi(y^k)) \label{eq:extragradient_x}
\end{align}
\end{subequations}
where $P_\XX(z) = \operatorname{argmin}_{x \in \XX}\,\|x - z\|^2$ denotes the Euclidean projection onto $\XX$.
Since $\phi$ is affine with Lipschitz constant $L = \|F\|_2$, the iterations $y^k$ and $x^k$ in~\eqref{eq:extragradient} are guaranteed to converge to a v-GNE whenever $\alpha < 1/L$~\cite{Kor76}. 
Note that,
in the presence of general linear constraints defining $\XX$, the method requires solving two QPs 
per iteration to solve the least-distance problems~\eqref{eq:extragradient}.

To avoid the need of solving two QPs per iteration, as proposed by the classical extragradient method, we can use, for example, 
the Douglas-Rachford splitting method~\cite{LionsMercier1979}, which requires only one projection per iteration\footnote{We took the extragradient method as a baseline as DR performed worse in our experiments.}. Moreover, to obtain an approach which requires no projection except a clipping step, we can apply the extragradient method to the so called augmented game (EG-aug), obtained by extending the game with a dual player as in~\cite{TatKam19}. For this purpose, we introduce the Lagrangian function for each primal player $i\in[N]$ as follows: 
\[L_i(x,\lambda) = J_i(x) + \lambda^\T (Ax - b) + \nu^\T (Ex - e).\]
The cost of the corresponding dual player is defined as 
\[g(x,\lambda,\nu)=-\lambda^\T (Ax - b) - \nu^\T (Ex - e).\] 
Thus, the augmented game is defined as a game with $N$ primal players minimizing local Lagrangians for the primal variables in $x$ and one dual player who minimizes $g(x,\lambda,\nu)$ over $\nu\in\rr^q$, and $\lambda\in\rr^m_{+}$. As the resulting augmented game inherits monotonicity and Lipschitz continuity of the pseudogradient (due to the structure of the original game), we can apply the extragradient method to the augmented game, using an appropriate step size $\beta < 1/L_{\rm aug}$, where $L_{\rm aug}$ is the Lipschitz constant of the pseudogradient of the augmented game. The resulting iterations read: 
\begin{subequations}
\label{eq:extragradient_aug}
\begin{align}
    \bar x^k &= x^k - \beta\bigl(F x^k + f + A^\T \lambda^k + E^\T \nu^k\bigr)\\
    \bar\lambda^k &= \bigl[\lambda^k - \beta (b - A x^k)\bigr]_+\\
    \bar\nu^k &= \nu^k - \beta (e - E x^k)\\
    x^{k+1} &= \bar x^k - \beta\bigl(F \bar x^k + f + A^\T \bar\lambda^k + E^\T \bar\nu^k\bigr)\\
    \lambda^{k+1} &= \bigl[\bar\lambda^k - \beta (b - A \bar x^k)\bigr]_+\\
    \nu^{k+1} &= \bar\nu^k - \beta (e - E \bar x^k)
\end{align}
\end{subequations}
where $(\bar x^k,\bar\lambda^k,\bar\nu^k)$ plays the role of the extrapolated point $y^k$ in~\eqref{eq:extragradient}.  We note that the projection in this case is replaced by a simple clipping of the dual variable $\lambda$ to the nonnegative orthant, which does not require solving extra QPs for projection calculations.

\section{Quadratic Programming Reformulation}
%\subsection{KKT conditions}
We now characterize v-GNEs through the KKT conditions of the players' optimization problems, which will serve as the starting point for the QP reformulation developed in this section. Since each cost $J_i$ is convex and the shared constraint set $\XX$ is polyhedral, a point $x^\star$ solves the variational inequality~\eqref{eq:VI}, and hence is a v-GNE, if and only if there exists a vector of Lagrange multipliers $\lambda^\star \in \rr^m$, $\nu^\star\in\rr^q$ associated with the shared constraints and common to all players, such that the pair $(x^\star,\lambda^\star,\nu^\star)$ satisfies the joint Karush--Kuhn--Tucker (KKT) conditions of the $N$ optimization problems~\eqref{eq:QPi}; see~\cite[Theorem 12.1]{NW06}. These conditions are given by
\begin{subequations}
\begin{align}
  F \xstar + f + A^\T \lstar + E^\T \nstar &= 0 \label{eq:KKT_stat}\\
  E \xstar &= e                \label{eq:KKT_eq}\\
  A \xstar &\leq b              \label{eq:KKT_prim}\\
  \lstar   &\ge 0              \label{eq:KKT_dual}\\
  (\lstar)^\T (A \xstar - b)   &= 0 \label{eq:KKT_comp}
\end{align}
\label{eq:KKT}%
\end{subequations}
where $\nstar \in \rr^q$ is the shared dual vector for the equality constraints and is unrestricted in sign.

Next, we exploit the joint KKT conditions~\eqref{eq:KKT} to derive and analyze a quadratic program whose solution set coincides with the set of v-GNEs of the considered game.
Let us denote by $F_s = \tfrac{1}{2}(F + F^\T)$ the \emph{symmetric part} of $F$. In the special case $F = F_s$,~\eqref{eq:QPi} is a {\it potential game}, and its variational GNEs are trivially the solutions of the following QP
\begin{equation}
    \begin{aligned}
    \min_{x_i} & \tfrac{1}{2} x^\T F x + f^\T x = \tfrac{1}{2} x^\T F_s x + f^\T x\\
    \textrm{s.t. }& Ax \leq b,\ Ex = e
    \end{aligned}
    \label{eq:QP-potential}
\end{equation}
whose KKT conditions are given by~\eqref{eq:KKT}. In general, however, the pseudogradient matrix $F$ is not symmetric, so the game is not a potential game, and it is a priori unclear whether the v-GNE problem admits any QP reformulation at all. The following result shows that it does, in the primal-dual variables $(x,\lambda,\nu)$.

\begin{lemma}
\label{lem:QP_reformulation}
Consider the following quadratic programming problem 
\begin{subequations}
\begin{align}
  \min_{x,\lambda,\nu}\quad &  x^\T F_s x + f^\T x + b^\T \lambda + e^\T \nu\label{eq:QP_obj}\\
  \textrm{s.t.}\quad & F x + f + A^\T \lambda + E^\T \nu = 0\label{eq:QP_stat}\\
            & E x = e \label{eq:QP_eq}\\
            & A x \leq b \label{eq:QP_prim}\\
            & \lambda \geq 0 \label{eq:QP_lambda}
\end{align}
\label{eq:QP}%
\end{subequations}
If~\eqref{eq:QP} is infeasible, then the LQ-GNEP~\eqref{eq:QPi} admits no v-GNE.
If instead~\eqref{eq:QP} is feasible, then, under Standing Assumption~\ref{assum:mon}, its optimal value is finite, nonnegative, and attained at some $(x^\star,\lambda^\star,\nu^\star)$; moreover, $x^\star$ is a v-GNE of~\eqref{eq:QPi} if and only if this optimal value is zero.
\end{lemma}

\begin{proof}
The feasible set of~\eqref{eq:QP}, defined by~\eqref{eq:QP_stat}--\eqref{eq:QP_lambda}, is precisely the set of triples $(x,\lambda,\nu)$ satisfying the KKT conditions~\eqref{eq:KKT} except for complementarity slackness~\eqref{eq:KKT_comp}. If this set is empty, no triple can satisfy~\eqref{eq:KKT} either, and since every v-GNE of~\eqref{eq:QPi} has an associated multiplier pair $(\lambda^\star,\nu^\star)$ satisfying~\eqref{eq:KKT}, the LQ-GNEP~\eqref{eq:QPi} admits no v-GNE. This proves the claim in the infeasible case.

Assume now that~\eqref{eq:QP} is feasible, and consider the problem of minimizing the complementarity slackness term in~\eqref{eq:KKT_comp}
\begin{equation}
\begin{aligned}
\min_{x,\lambda,\nu}\ &\lambda^\T (b - A x)\\
\textrm{s.t. }& \mbox{\eqref{eq:QP_stat}--\eqref{eq:QP_lambda}},
\end{aligned}
\label{eq:slack-minimization}
\end{equation}
which, by construction, shares the feasible set of~\eqref{eq:QP}. By substituting $A^\T \lambda=-F x - f -  E^\T \nu$ (from~\eqref{eq:QP_stat}) and $Ex=e$ (from~\eqref{eq:QP_eq}) into the objective of~\eqref{eq:slack-minimization}, we obtain, for every feasible $(x,\lambda,\nu)$,
 \begin{align}\label{eq:equiv}
    \lambda^\T (b - A x)&=\lambda^\T b + x^\T(F x + f + E^\T \nu) \cr
    &= \lambda^\T b + x^\T F x + f^\T x + e^\T \nu\cr
    &= x^\T F_s x + f^\T x + b^\T \lambda + e^\T \nu,
    \end{align}
i.e., the objective of~\eqref{eq:slack-minimization} coincides with that of~\eqref{eq:QP} on their common feasible set, the two problems are in fact the same problem.

      Since $Ax\leq b$ and $\lambda\geq 0$ at every feasible point (by~\eqref{eq:QP_prim} and~\eqref{eq:QP_lambda}), we have $\lambda^\T(b-Ax)\geq 0$, and hence, by~\eqref{eq:equiv}, the objective of~\eqref{eq:QP} is bounded below by zero on its feasible set. Moreover, under Standing Assumption~\ref{assum:mon} the objective of~\eqref{eq:QP} is a convex quadratic function ($F_s\succeq 0$), and its feasible set, defined by~\eqref{eq:QP_stat}--\eqref{eq:QP_lambda}, is polyhedral. A convex quadratic function that is bounded below on a polyhedron attains its infimum on that polyhedron (the classical Frank--Wolfe existence theorem for quadratic programs, see, e.g.,~\cite{BSS06}), so the optimal value of~\eqref{eq:QP} is finite, nonnegative, and attained at some $(x^\star,\lambda^\star,\nu^\star)$.

It remains to relate this optimal value to the existence of a v-GNE. If $(x^\star,\lambda^\star,\nu^\star)$ attains a zero optimal value, then, being feasible for~\eqref{eq:QP} and satisfying $(\lambda^\star)^\T(b-Ax^\star)=0$ by~\eqref{eq:equiv}, it satisfies the joint KKT conditions~\eqref{eq:KKT} of the game, and hence $x^\star$ is a v-GNE of~\eqref{eq:QPi}. Conversely, if $x^\star$ is a v-GNE of~\eqref{eq:QPi}, then there exist $\lambda^\star,\nu^\star$ such that $(x^\star,\lambda^\star,\nu^\star)$ satisfies~\eqref{eq:KKT}; this triple is in particular feasible for~\eqref{eq:QP}, and its objective value equals $(\lambda^\star)^\T(b-Ax^\star)=0$ by~\eqref{eq:KKT_comp}, so the (nonnegative) optimal value of~\eqref{eq:QP} must be zero.
\end{proof}

\begin{remark}
If the shared constraint set $\XX$ defined in~\eqref{eq:XX} is non-empty
and compact, then LQ-GNEP~\eqref{eq:QPi} admits a v-GNE~\cite{FP03}, and, thus, the QP~\eqref{eq:QP} always admits a solution with the zero optimal value.
Clearly, when the game is merely monotone ($F_s \succeq 0$), the QP problem~\eqref{eq:QP} is convex.
Moreover, if the game is strongly monotone ($F_s\succ 0$) and $\XX$ nonempty, a v-GNE solution
$\xstar$ exists and is unique~\cite[Thm.~2.3.3]{FP03}, and therefore the
QP~\eqref{eq:QP} has a solution $(\xstar,\lstar,\nstar)$ with zero optimal value.
\end{remark}
In the next section we provide the analysis of some methods solving the QP in~\eqref{eq:QP}. In particular, we will connect the convergence rate of the objective $x^\T F_s x + f^\T x + b^\T \lambda + e^\T \nu$ with the guarantees of an algorithm's iterate to be an $\varepsilon$-appGNE. To do so, we will use the following result. 
\begin{lemma}\label{lem:eGNE}
    Assume that $(\tx, \tl, \tn)$ is a feasible point with respect to the QP~\eqref{eq:QP} such that
    \[\tx^\T F_s \tx + f^\T \tx + b^\T \tl + e^\T \tn\le \varepsilon.\]
    Then $\tx$ is a $\varepsilon$-appGNE in the game defined by~\eqref{eq:QPi}.
\end{lemma}

\begin{proof}
Due to~\eqref{eq:equiv}, we conclude that
\begin{align}\label{eq:complGap}
\tl^\T (b - A \tx)\le\varepsilon.
\end{align}
Next, let us consider the primal gap function
\begin{align}\label{eq:Gap}
  \operatorname{Gap}(\tx)=\sup_{z\in\XX} \phi(\tx)^\T(\tx-z),  
\end{align}
where, as before, $\phi(x) = Fx+f$. 
From the stationarity condition in~\eqref{eq:QP_stat} and due to feasibility of $(\tx, \tl, \tn)$, we obtain $\phi(\tx)+A^\T\tl + E^\T \tn = 0$. Thus, $\phi(\tx)=-A^\T\tl - E^\T\tn$, and for any \(z\in\XX\),
\begin{align*}
\phi(\tx)^\T(\tx-z)
&=
-(\tl^\T A + \tn^\T E) (\tx-z)=
\tl^\T(Az-A\tx),
\end{align*}
where in the last equality we used the fact that $E (\tx-z) = 0$, as $\tx\in \XX$ and $z\in \XX$. 
Moreover, since $z\in\XX$, we have $Az\le b$. Together with
$\tl\ge 0$ and~\eqref{eq:complGap}, it follows that 
$\tl^\T(Az-A\tx) \le
\tl^\T(b-A\tx) \le
\varepsilon$.
Thus, we conclude that
\[
\operatorname{Gap}(\tx)
=
\sup_{z\in\XX}
\phi(\tx)^\T(\tx-z) =\sup_{z\in\XX} \tl^\T(Az-A\tx)
\le
\varepsilon.
\]
Fix a player $i$ and let $z_i$ be an admissible unilateral deviation, i.e., $(z_i,\tx_{-i})\in\XX$.
Let $z=(z_i,\tx_{-i})$.
As $z\in\XX$, it follows from the definition of the gap function that
$\phi(\tx)^\T(\tx-z)
\le
\operatorname{Gap}(\tx)
\le
\varepsilon
$.
Since $z$ differs from $\tx$ only in the $i$-th component, $
\phi(\tx)^\T(\tx-z) =
\bigl(\nabla_{x_i}J_i(\tx)\bigr)^\T(\tx_{-i}-z_i)$.
Hence,
$\bigl(\nabla_{x_i}J_i(\tx)\bigr)^\T(\tx_{-i}-z_i)
\le
\varepsilon$.
Due to convexity of \(J_i(\cdot,\tx_{-i})\) (see Standing Assumption~\ref{assum:mon}), we have
\[
J_i(\tx_{i},\tx_{-i})-J_i(z_i,\tx_{-i})
\le
\bigl(\nabla_{x_i}J_i(\tx)\bigr)^\T(\tx_{i}-z_i).
\]
By combining the last two inequalities, we obtain
$
J_i(\tx_{i},\tx_{-i})
\le
J_i(z_i,\tx_{-i}) + \varepsilon
$,
which, according to Definition~\ref{def:appGNE}, implies that $\tx$ is $\varepsilon$-appGNE.
\end{proof}

\section{Solution Methods}
\label{sec:solution_methods}
The QP problem~\eqref{eq:QP} can be solved by any off-the-shelf convex QP solver that can handle a singular Hessian matrix. While a few QP solvers do not require strong convexity~\cite{GC24,SBGBB20}, %Mosek, 
many others either require Hessian inversion explicitly (e.g., dual-QP methods) or are not numerically robust in the case where the Hessian is not positive definite. In this section we briefly discuss some efficient solution methods 
to cope with the lack of strong convexity of the QP problem~\eqref{eq:QP} by leveraging QP solvers 
designed for strongly convex QPs. 

\subsection{QP Regularization}
Let us analyze the effect of adding a regularization in the QP~\eqref{eq:QP}:  
\begin{align}\label{eq:QP_r}
  \min_{x,\lambda,\nu}\quad &  x^\T F_s x + f^\T x + b^\T \lambda + e^\T \nu \cr
  &\qquad\qquad + \tfrac{\rho}{2}(\|x\|^2 + \|\lambda\|^2 + \|\nu\|^2)\cr
  \textrm{s.t.}\quad & F x + f + A^\T \lambda + E^\T \nu = 0\cr
            & E x = e \cr
            & A x \leq b \cr
            & \lambda \geq 0 
\end{align}
which makes the resulting QP strongly convex when $\rho>0$.

\begin{proposition}
\label{prop:eps-GNE-eq}
Let $\rho > 0$, and assume that a v-GNE primal-dual triple $(\xstar,\lstar,\nstar)$ of~\eqref{eq:KKT} exists.
Let $(x(\rho),\lambda(\rho),\nu(\rho))$ be the unique solution of the strongly convex QP in~\eqref{eq:QP_r}. 
Then $x(\rho) \in \XX$, and $x(\rho)$ is an $\varepsilon(\rho)$-appGNE of~\eqref{eq:QPi}
according to Definition~\ref{def:appGNE} with
\begin{equation}
  \varepsilon(\rho) = \lambda(\rho)^\T(b - Ax(\rho))
  \;\leq\;
  \frac{\rho}{2}\bigl(\|\xstar\|^2 + \|\lstar\|^2 + \|\nstar\|^2\bigr).
  \label{eq:eps-bound-eq}
\end{equation}
In particular, $\varepsilon(\rho) = \mathcal{O}(\rho)$ as $\rho \to 0$.
\end{proposition}

\begin{proof}
Let $\varepsilon(\rho) = \lambda(\rho)^\T(b-Ax(\rho))$. 
By Lemma~\ref{lem:QP_reformulation}, we have that
\begin{align}\label{eq:eps}
  \varepsilon(\rho) &= \lambda(\rho)^\T(b-Ax(\rho))= x(\rho)^\T F_s x(\rho)\cr &+f^\T x(\rho)+b^\T\lambda(\rho)+e^\T\nu(\rho).
\end{align}
Hence the regularized objective at $(x(\rho),\lambda(\rho),\nu(\rho))$ equals
$\lambda(\rho)^\T(b-Ax(\rho)) + \tfrac{\rho}{2}(\|x(\rho)\|^2+\|\lambda(\rho)\|^2+\|\nu(\rho)\|^2)$.
Since $(\xstar,\lstar,\nstar)$ is feasible for the regularized QP with zero complementarity gap, its objective value equals $\tfrac{\rho}{2}(\|\xstar\|^2+\|\lstar\|^2+\|\nstar\|^2)$.
Optimality of $(x(\rho),\lambda(\rho),\nu(\rho))$ therefore gives:
\[
    \begin{aligned}
  \varepsilon(\rho) + \tfrac{\rho}{2}(\|x(\rho)\|^2&+\|\lambda(\rho)\|^2+\|\nu(\rho)\|^2)\\
  &\leq \tfrac{\rho}{2}\bigl(\|\xstar\|^2+\|\lstar\|^2+\|\nstar\|^2\bigr),
    \end{aligned}
\]
and the bound~\eqref{eq:eps-bound-eq} follows since the second term on the left is non-negative. 

Finally taking into account the second equality in~\eqref{eq:eps} and using Lemma~\ref{lem:eGNE}, we conclude that $x(\rho)$ is an $\varepsilon(\rho)$-appGNE of the game-theoretic problem~\eqref{eq:QPi}. 
\end{proof}
Note that the bound~\eqref{eq:eps-bound-eq} is an {\it a priori} estimate of the suboptimality of $x(\rho)$ as an $\varepsilon(\rho)$-appGNE for a given problem, and it is not tight in general. However, the actual suboptimality parameter $\varepsilon(\rho)$ is immediately available {\it a posteriori} from the solution of the regularized QP as $\varepsilon(\rho) = \lambda(\rho)^\T(b-Ax(\rho))$.

\begin{remark}
Although the LQ-GNEP~\eqref{eq:QPi} may possess multiple v-GNEs when the game is only merely monotone, the regularized solution $(x(\rho),\lambda(\rho),\nu(\rho))$ is unique for every $\rho>0$, since~\eqref{eq:QP_r} is strongly convex. Furthermore, since $F_s\succeq 0$ and $\XX$ is polyhedral, the optimal solution set of the unregularized QP~\eqref{eq:QP} is convex, so it admits a unique element $(x^\star_{\min},\lambda^\star_{\min},\nu^\star_{\min})$ of minimum Euclidean norm. A standard argument for Tikhonov-regularized convex programs (see, e.g.,~\cite{FP03}) shows that $(x(\rho),\lambda(\rho),\nu(\rho)) \to (x^\star_{\min},\lambda^\star_{\min},\nu^\star_{\min})$ as $\rho\to 0$; in particular, $x(\rho)$ converges to the minimum-norm v-GNE of~\eqref{eq:QPi} among the game's possibly infinitely many equilibria.
\end{remark}

In the following subsection we formulate a proximal-point method to state an asymptotic convergence to the set of v-GNEs.

\subsection{Accelerated Proximal-Point QP Iterations}
\label{sec:proximal-point}
In order to solve the QP~\eqref{eq:QP} by using a strongly convex QP solver, we apply
the accelerated proximal-point (APP) iterations  to the following non-smooth convex problem~\cite{Gul92}:
\begin{equation}
\min_{z} \Psi(z)=
x^\T F_sx+f^\T x+b^\T\lambda+e^\T\nu+\II(z)
\label{eq:non-smooth}
\end{equation}
where $z = (x, \lambda, \nu)\in\rr^{n+m+q}$ and $\II:\rr^{n+m+q}\to\rr\cup\{+\infty\}$ is the indicator function of the feasible set defined by the convex constraints in~\eqref{eq:QP}.
Let \(z^0=(x^0,\lambda^0,\nu^0)\in\rr^{n+m+q}\), let
\(\{\gamma_k\}_{k\ge0}\) be positive proximal parameters, and set
\(\theta_0=1\). First compute
\[
z^1\in\arg\min_{z}
\left\{
\Psi(z)+\frac{1}{2\gamma_0}\|z-z^0\|^2
\right\}.
\]
For \(k=1,2,\ldots\), write $y^k=(\bar x^k,\bar\lambda^k,\bar\nu^k)$,
$z^{k+1}=(x^{k+1},\lambda^{k+1},\nu^{k+1})$,
and perform
\begin{subequations}
\begin{align}
\frac{\theta_k^2}{\gamma_k}
&=
(1-\theta_k)\frac{\theta_{k-1}^2}{\gamma_{k-1}},
\qquad
\theta_k\in(0,1]
\label{eq:guler_theta}
\\
y^k
&=
z^k+
\theta_k\left(\frac{1}{\theta_{k-1}}-1\right)
(z^k-z^{k-1})
\label{eq:guler_extrapolation}
\\
z^{k+1}
&\in
\arg\min_{x,\lambda,\nu}
\Bigg\{
x^\T F_sx+f^\T x+b^\T\lambda+e^\T\nu
\nonumber\\
&\qquad
+
\frac{1}{2\gamma_k}
\left(
\|x-\bar x^k\|^2
+
\|\lambda-\bar\lambda^k\|^2
+
\|\nu-\bar\nu^k\|^2
\right)
\Bigg\}
\nonumber\\
&\quad
\text{s.t.~\eqref{eq:QP_stat}--\eqref{eq:QP_lambda}}
\label{eq:guler_prox_qp}
\end{align}
\label{eq:prox-QP}%
\end{subequations}
where the proximal parameter $\frac{1}{\gamma_k}> 0$ acts as the regularization parameter of the QP.
According to the result in~\cite[Theorem 2.2]{Gul92}, the following holds for the $k$th iterate of the APP, in the case a solution $z^*$ to the QP~\eqref{eq:QP} exists: 
\[\Psi(z_k) - \Psi(z^*) = \mathcal O\left(\frac{1}{
\left(\sum_{j=0}^{k-1}\sqrt{\gamma_j}\right)^2
}
\right).
\]
In particular, if \(\gamma_j\equiv\gamma>0\), then
\[\
\Psi(z_k)-\Psi^\star
=
\mathcal O\left(\frac{1}{k^2}\right).
\]
Taking into account that $\Psi(z_k) = x_k^\T F_sx_k+f^\T x_k+b^\T\lambda_k+e^\T\nu_k$ and $\Psi(z^*) = 0$, we conclude by Lemma~\ref{lem:eGNE} that the following result holds. 
\begin{proposition}\label{prop:AccPP}
    The iterate $x_k$ of the accelerated proximal point procedure in~\eqref{eq:prox-QP} with some constant $\gamma_k=\gamma>0$ is an $\mathcal{O}\!\left(\frac{1}{k^2}\right)$-appGNE in the game defined by~\eqref{eq:QPi}.
\end{proposition}

Note that, in quantifying the convergence rate of the iterations, we assume that each QP in~\eqref{eq:guler_prox_qp} is solved in a bounded (although, possibly large) number of operations, such as by an active-set method that
has a finite worst-case complexity due to the finite number of possible active sets. 

\begin{remark}
While the proximal parameters $\gamma_k$ could in principle be chosen time-varying with a potential improvement in theoretical guarantees for the approximation rate, letting the regularization parameter $1/\gamma_k$ increase with $k$ is problematic in practice: as $\gamma_k\to 0$, the term $\tfrac{1}{2\gamma_k}\|\cdot\|^2$ in~\eqref{eq:guler_prox_qp} dominates the KKT system of the subproblem, whose entries then span increasingly disparate orders of magnitude (the $F$-, $A$-, and $E$-related rows staying $\mathcal O(1)$ while the regularization rows scale as $\mathcal O(1/\gamma_k)$). This ill-conditions the linear systems solved at each step of an active-set method and can make it numerically difficult to take an accurate step, even though the subproblem itself remains strongly convex. This motivates our choice of a constant proximal parameter.
\end{remark}

\subsection{Accelerated Projected Gradient}
\label{sec:APG}
Problem~\eqref{eq:QP} can also be solved by the accelerated proximal gradient (APG) method~\cite{Nes83}
applied to the non-smooth convex problem~\eqref{eq:non-smooth}. Let, as before, $z=(x,\lambda,\nu)$ and $y=(\bar{x},\bar{\lambda},\bar{\nu})$ be the extrapolated variable, $z,y\in\rr^{n+m+q}$. 
A valid Lipschitz constant of the gradient of the smooth part of the objective in~\eqref{eq:non-smooth} is $L=2\|F_s\|_2$.
Given the constant stepsize $1/L$, the procedure is defined by the following iterations:
\begin{subequations}
\begin{align}
\nabla f_k&=[f+(F_s(\bar{x}^k))^\T\ b^\T\ e^\T]^\T\\
z^{k+1}&=\arg\min_z \|z-(y^k-\frac{1}{L} \nabla f_k)\|_2^2\\
  &\textrm{s.t.}\ z=(x,\lambda,\nu)\ \mbox{satisfies~\eqref{eq:QP_stat}--\eqref{eq:QP_lambda}}\label{eq:APG_prox}\\
\theta_{k+1}&=\frac{\theta_k\left(\sqrt{\theta_k^2+4}-\theta_k\right)}{2}\in(0,1]\label{eq:APG_theta}\\
y^{k+1}&=z^{k+1}+\theta_{k+1}\left(\frac{1}{\theta_k}-1\right)(z^{k+1}-z^k)
\end{align}\label{eq:APG}%
\end{subequations}
starting at $k=0$ from an initial point $z^0$ satisfying the constraints~\eqref{eq:QP_stat}--\eqref{eq:QP_lambda}, $\theta_0=1$, and $y^0=z^0$, with~\eqref{eq:APG_theta} being the solution of~\eqref{eq:guler_theta} for constant $\gamma_k$.
Therefore, as seen for proximal-point iterations, the method requires solving the strongly convex QP~\eqref{eq:APG_prox} at each iteration, possibly warm-starting the QP solver with the active-set obtained at the previous iteration. 
The violation of complementarity slackness, easily retrieved from~\eqref{eq:QP_obj}, can be used as a stopping criterion for the method.

Let $z^\star=(x^\star,\lambda^\star,\nu^\star)$ be a v-GNE of~\eqref{eq:QPi}, 
and, therefore, a solution of QP~\eqref{eq:QP} with zero optimal value $f^\star$.
As proved in~\cite[Theorem 4.4]{BT09}, the iterates generated by the APG method above satisfy
\[
  (\lambda^k)^\T(b-Ax^k)
  \leq
  \frac{4\|F_s\|_2\,\|z^1-z^\star\|^2}{(k+1)^2}
  =\mathcal{O}\!\left(\frac{1}{k^2}\right).
\]
Taking the relation above into account, we can use again Lemma~\ref{lem:eGNE} to formulate the following result. 
\begin{proposition}\label{prop:AccGD}
    The iterate $x^k$ of the accelerated projected gradient procedure in~\eqref{eq:APG} is an $\mathcal{O}\!\left(\frac{1}{k^2}\right)$-appGNE in the game defined by~\eqref{eq:QPi}.
\end{proposition}
Note that restarting the extrapolation coefficients $\theta_k$ adaptively, as suggested in~\cite{OC15}, may further improve the convergence of the method in practice.

\subsection{Comparison of Convergence Rates}
We close this section by comparing the guarantees of Propositions~\ref{prop:AccPP}--\ref{prop:AccGD} for the proposed QP-based methods with the guarantee available for the extragradient method~\eqref{eq:extragradient}.

Unlike the accelerated proximal-point and accelerated projected-gradient methods, which by Propositions~\ref{prop:AccPP}--\ref{prop:AccGD} produce a \emph{last-iterate} $\mathcal O(1/k^2)$-appGNE under mere monotonicity of $\phi$ alone, a comparable last-iterate guarantee is not available for the extragradient method~\eqref{eq:extragradient} in this setting. A rate is only known for the \emph{ergodic average} of its iterates, and it additionally requires the shared constraint set $\XX$ to be compact~\cite[Thm.~3.2 and Sec.~5]{Nemirovski04}. To state it, let us define the restricted Minty gap over the compact set $\XX$,
\[m(x) = \sup_{z\in \XX} \phi(z)^\T(x-z)
\]
which, in contrast with the gap function $\operatorname{Gap}(\cdot)$ in~\eqref{eq:Gap}, evaluates the pseudogradient $\phi$ at the candidate point $z$ rather than at $x$; by monotonicity of $\phi$ (Standing Assumption~\ref{assum:mon}), $m(x)\le \operatorname{Gap}(x)$ for every $x\in\XX$.

Let $\bar y^K = \frac{1}{K} \sum_{k=0}^{K-1} y^k$ denote the ergodic average of the extrapolated iterates $y^k$ in~\eqref{eq:extragradient}. Then, by~\cite[Thm.~3.2]{Nemirovski04}, the extragradient method with a constant stepsize $\alpha \le 1/(\sqrt{2}L)$, where $L=\|F\|_2$ is the Lipschitz constant of $\phi$, satisfies
\[m(\bar y^K)
=
\sup_{z\in \XX}\phi(z)^\T(\bar y^K-z)
\le
\frac{1}{2\alpha K}\sup_{z\in \XX}\|x^0-z\|^2
\]
so that $m(\bar y^K)=\mathcal O\!\left(\frac{1}{K}\right)$. Since $m(\cdot)$ only bounds the weaker Minty gap rather than the gap function $\operatorname{Gap}(\cdot)$ used in Lemma~\ref{lem:eGNE} to certify an appGNE, translating this bound into an actual suboptimality guarantee for $\bar y^K$ incurs the loss of a square root: $\bar y^K$ is only guaranteed to be an $\mathcal O\!\left(\frac{1}{\sqrt K}\right)$-appGNE of~\eqref{eq:QPi}, see~\cite{TatNed25}.

Thus, in view of Propositions~\ref{prop:AccPP}--\ref{prop:AccGD}, the accelerated proximal-point and accelerated projected-gradient methods yield a strictly more accurate approximation of a v-GNE than the extragradient method for a given number of iterations ($\mathcal O(1/k^2)$ versus $\mathcal O(1/\sqrt k)$). Moreover, they provide this guarantee for the last iterate rather than for an ergodic average, and they do not require the joint action set $\XX$ to be compact.

\section{Special Classes of Linear-Quadratic GNEPs}
In this section, we consider some special cases of the game-theoretic problem defined in~\eqref{eq:QPi}. 

\subsection{Invertible $F$ and Variable Elimination}
Although good QP solvers can handle equality constraints efficiently,
we now show that, under additional assumptions, it is possible to eliminate $x$ and $\nu$ from the problem in advance and get a QP which is dependent merely on $\lambda$. In this way, we can reduce the number of variables to $m$, at the price of potentially destroying sparsity structures that may exist in the original problem~\eqref{eq:QPi}.  

\begin{lemma}
Let $F+F^\T\succeq 0$, $F$ invertible, and, when $q>0$, assume $\operatorname{rank}(E) = q$ and $M = E F^{-1} E^\T$ invertible. Then, the v-GNE problem~\eqref{eq:QPi} can be reformulated as the following convex QP in $\lambda$:
\begin{equation}
    \begin{aligned}
    \min_{\lambda}\ & \lambda^\T H_s\lambda + q_0^\T\lambda\\
    \textrm{s.t.}\ & H\lambda + q_0\geq 0,\quad \lambda\geq 0
    \end{aligned}
    \label{eq:QP_dual}
\end{equation}
where $H_s=(H+H^\T)/2$ and $H$ and $q_0$ are defined as:
\[
\begin{aligned}
H&=APA^\T,\quad P=G-GE^\T M^{-1}EG,\quad G=F^{-1}\\
q_0&=b-Ax_0,\quad x_0=-Pf+GE^\T M^{-1}e.
\end{aligned}
\]
The v-GNE $x^\star$ of~\eqref{eq:QPi} and the associated dual vector $\nu^\star$ 
can be recovered from the optimal solution $\lambda^\star$ of~\eqref{eq:QP_dual} as
\[
x^\star = x_0-PA^\T\lambda^\star,\quad \nu^\star = 
-M^{-1}(EG(f+A^\T\lambda^\star)+e).
\]
\label{lem:QP_dual}
\end{lemma}

\begin{proof}
See Appendix~\ref{app:QP_dual}.
\end{proof}

\begin{remark}
Assuming that $\operatorname{rank}(E) = q$ is with no loss of generality. In fact, in the case $E$ has linearly dependent rows, one can perform a QR decomposition with pivoting $\Pi [E\ -e] = Q \left[\begin{smallmatrix} R_{11} & R_{12} \\ 0 & R_{22} \end{smallmatrix}\right]$. If $R_{22}\neq 0$, the game has no solution as $\{x:\ Ex=e\}=\emptyset$. Otherwise, we replace
the equality constraints by $[R_{11}\ R_{12}]\left[\begin{smallmatrix} x' \\ 1 \end{smallmatrix}\right]=0$.
\end{remark}

The following lemma shows that strict monotonicity is a sufficient condition for the assumptions of the previous lemma to hold.

\begin{lemma}
Let the GNEP~\eqref{eq:QPi} be strictly monotone, i.e., $F+F^\T \succ 0$, and $E$ be full row-rank. Then, $F^{-1}$ exists and $M = E F^{-1} E^\T$ is invertible.
\end{lemma}
\begin{proof}
Suppose by contradiction that $F$ is not full rank, i.e., $F x = 0$ for some $x \neq 0$.
Then $0=x^\T F x = x^\T F_s x$, contradicting $F_s \succ 0$, and hence $F^{-1}$ must exist.
Consider any $x\neq 0$ and let $y=Ex$, where $y\neq 0$ as $E$ is full row-rank. 
Then, $x^\T Mx=y^\T F^{-1} y = (F^{-1}y)^\T F (F^{-1}y) = (F^{-1}y)^\T F_s (F^{-1}y) > 0$,
for all $x\neq 0$, and hence $M$ positive definite and therefore invertible.
\end{proof}

We remark that strict monotonicity is not necessary for the assumptions of Lemma~\ref{lem:QP_dual} to hold. For example, consider a two-player game with $n_1=n_2=1$, $Q^{(1)} = \left[\begin{smallmatrix} 1 & 0 \\ 0 & 0 \end{smallmatrix}\right]$, $Q^{(2)} = \left[\begin{smallmatrix} 4 & 2 \\ 2 & 1 \end{smallmatrix}\right]$, $c^{(1)}=c^{(2)}=0$, and no constraints. The pseudogradient matrix $F=\left[\begin{smallmatrix} 1& 0\\2&1 \end{smallmatrix}\right]$ is invertible, but its symmetric part $F_s = \left[\begin{smallmatrix} 1 & 1 \\ 1 & 1 \end{smallmatrix}\right]$ has eigenvalues $0$ and $2$.

Regarding the computations involved in constructing the reduced QP~\eqref{eq:QP_dual}, a numerically efficient method is to first find an LU decomposition with pivoting $\Pi F=LU$ of $F$, $F\in\rr^{n\times n}$, then solve $FY=E^\T$ for $Y$ by forward and backward substitution and compute $M=E Y$. 
Then, one solves $F^{\T}V=E^\T$ and sets $Z=V^\T$ using the same LU decomposition.
By computing a second LU decomposition with pivoting $\Pi' M = L_MU_M$, $M\in\rr^{q\times q}$,
we solve $MC=Z$ for $C$ by forward and backward substitution. Finally, we compute $Q=I-E^\T C$ and solve $FP=Q$ for $P$. In the case of strictly monotone games, as $F$ is positive definite, LU decompositions can be replaced by Cholesky factorizations.

The proximal-point iterations described in Section~\ref{sec:proximal-point} can be applied to the reduced QP~\eqref{eq:QP_dual} as well, leading to a sequence of strongly convex QPs in $\lambda$ alone. In this case, 
the suboptimality parameter $\varepsilon$ is immediately available from the optimal cost of the QP as
$\varepsilon^{k+1} = (\lambda^{k+1})^\T (H_s\lambda^{k+1} + q_0)-\frac{\rho}{2}\|\lambda^{k+1}-\lambda^{k}\|_2^2$
and can be used as a stopping criterion for the method.

\section{Game-Theoretic MPC}
\label{sec:mpc}
Game-theoretic linear MPC problems with $N$ agents can be cast as LQ-GNEPs with shared constraint (see, e.g.,~\cite[Section 5.2]{Bem26}).
Let us assume that each agent decides a subset $u_i\in\rr^{n_{u_i}}$ of the input vector $u=[u_1^\T,\ldots,u_N^\T]^\T$, $u\in\rr^{n_u}$, $n_u=\sum_{i=1}^N n_{u_i}$, and makes predictions over a finite horizon $T$ according to the following linear time-invariant dynamics:
\begin{subequations}
\label{eq:mpc_dyn}
  \begin{align}
  x^i_{k+1} &= {\mathcal A}^ix^i_k + {\mathcal B}^iu_k \label{eq:mpc_dyn_state}\\
    y^i_k &= {\mathcal C}^ix_k 
  \end{align}
\end{subequations}
where $x^i_k\in\rr^{n_x}$ is the state vector used by the agent to make an agent-specific prediction $y^i_k$ of the system
output $y_k$, $y_k\in\rr^{n_y}$, and ${\mathcal B}^i=[{\mathcal B}^i_1,\ldots,{\mathcal B}^i_N]$. In general, each agent can have a different model $({\mathcal A}^i,{\mathcal B}^i,{\mathcal C}^i)$ of the system dynamics
\begin{equation}
\label{eq:mpc_dyn-process}
  x(t{+}1) = {\mathcal A}x(t) + {\mathcal B}u(t), \quad y(t) = {\mathcal C}x(t)
\end{equation}
that generates the data, where ${\mathcal A}\in\rr^{n_x\times n_x}$, ${\mathcal B}\in\rr^{n_x\times n_u}$, and ${\mathcal C}\in\rr^{n_y\times n_x}$ are the true system matrices, or a common model of~\eqref{eq:mpc_dyn-process}, depending on the application. The state $x(t)$ of the system can be either assumed fully observed by all agents at each time step $t$, or provided to each agent $i$ by a centralized estimator, e.g., a Kalman filter, that fuses the measurements of all agents, or estimated by each agent $i$ by using its own model and local, plus possibly shared, measurements. 

Consider the following extension of the linear GT-MPC setup in~\cite{Bem26}, where each agent minimizes a finite-horizon cost over its own input trajectory, subject to input and output constraints. By letting $\Delta u(t)=u(t)-u(t{-}1)$ denote the input increments, 
each agent~$i$ minimizes over
$(\Delta u^i_0,\ldots,\Delta u^i_{T{-}1},\varepsilon_i)$ the finite-horizon cost
\begin{align}\label{eq:mpc_cost}
  J_i = \sum_{k=0}^{T-1}&\bigl[(y^i_{k+1}-r^i(t))^\T Q_y^{(i)}(y^i_{k+1}-r^i(t))\notag\\
  &{}+ (\Delta u^i_k)^\T Q_{\Delta u}^{(i)}\,\Delta u^i_k\bigr]
  + Q_{\varepsilon_1}\varepsilon_i + Q_{\varepsilon_2}\varepsilon_i^2
\end{align}
where $r^i(t)\in\rr^{n_y}$ is the current output set-point agent $i$ wishes the output $y(t)$ of the system to track, subject to the dynamics~\eqref{eq:mpc_dyn}.
In~\eqref{eq:mpc_cost}, $\epsilon_i$ is a per-agent scalar slack variable used to
soften the output constraints
\begin{equation}\label{eq:mpc_output_constraints}
  y_{\min}^i - \textstyle\sum_{j=1}^N\epsilon_j\,\mathbf{1} \le y_{k+1}^i \leq
  y_{\max}^i + \textstyle\sum_{j=1}^N\epsilon_j\,\mathbf{1}
\end{equation}
imposed by each agent on the predicted output vector $y^i_{k+1}$, where $y_{\min}^i,y_{\max}^i$ can be either common or agent-specific. Local input constraints
\begin{equation}\label{eq:mpc_local}
\begin{aligned}
  &u_{\min}^i \leq u^i_k \leq u_{\max}^i,&\epsilon_i\geq 0\\
  &\Delta u_{\min}^i \leq \Delta u^i_k \leq \Delta u_{\max}^i
\end{aligned}
\end{equation}
are also imposed by each agent for $k=0,\ldots,T{-}1$. 

Each agent's prediction is initialized at the corresponding initial state $x^i_0$. The set-points $r^i(t)$ can be either conflicting or coincident; in the special case where each agent $i$ only weights a subset
$y_i$ of the output vector, matrix $Q_y^{(i)}$ has nonzero entries only in the rows and columns corresponding to $y_i$, so we can assume coincident set-points $r^i(t)=r(t)$ for all $i$ when the subsets $y_i$ are disjoint. 

By treating $(\Delta u^i_0,x^i_1,\ldots,\Delta u^i_{T{-}1},x^i_T,\epsilon_i)$ as the decision vector of agent $i$, the problem above can be cast as an LQ-GNEP with linear equality and inequality constraints of the form~\eqref{eq:QPi}. Solving such a QP clearly requires aggregating the information about the initial states, the system dynamics, the cost matrices, and the constraints of all agents. 

Note that the inequality constraints due to bounds on inputs and outputs are {\it local},
as they only involve $u^i_k$ and $x^i_{k+1}$, respectively, while
the dynamic constraints~\eqref{eq:mpc_dyn_state} are {\it shared} by all agents, as, in general, each predicted state $x^i_{k+1}$ depends on the entire input vector $u_k$. The variational nature of the game imposes that the dual variables associated with the dynamic constraints are shared by all agents.

The QP reformulation~\eqref{eq:QP} is very sparse, due to the sparsity of each agent's cost and  constraints with respect to the entire decision vector of the game. A {\it condensing} of the MPC problem can be obtained by eliminating the state variables $x^i_k$ from the problem, at the price of sacrificing much of the sparsity of the problem. While this eliminates the shared equality constraints~\eqref{eq:mpc_dyn_state}, the resulting problem is still an LQ-GNEP with shared inequality constraints due to~\eqref{eq:mpc_output_constraints}, and can be solved by the same methods described in this paper.

\section{Numerical Results}
We compare the performance of the methods described in Section~\ref{sec:solution_methods} for solving $N$-player LQ-GNEPs on randomly generated instances and in a game-theoretic model predictive control (MPC) setting. 
We exclude two further classes of methods from this comparison: MILP-based methods that solve the KKT system~\eqref{eq:KKT} directly, as described in~\cite{Bem26}, since they are mainly designed for {\it non-variational} and possibly non-monotone LQ-GNEPs; and KKT-residual methods, since they rely on nonconvex optimization and would be inefficient for the class of GNEPs considered here.

All numerical tests are run in Python 3.11 on a MacBook Pro with Apple M4 Max (16 CPU cores), using the DAQP solver~\cite{ABA22b} for solving strongly convex QP problems
with warm start of the active set. We also use the QP solver Clarabel~\cite{GC24} to solve the QP reformulation~\eqref{eq:QP} with singular Hessian matrix. The code for reproducing the numerical results is available at \url{https://github.com/bemporad/QP-LQGNE}.

\subsection{Randomly Generated Merely-Monotone GNEPs}
% example_lq_vgne_comparison.py
We consider randomly generated instances of $N$-player merely-monotone LQ-GNEPs in which each player $i$ owns $\bar n$ decision variables $x_i \in \rr^{\bar n}$, i.e., the joint vector $x = (x_1,\ldots,x_N) \in \rr^{N\bar n}$. Each matrix $Q_i \in \rr^{N\bar n \times N\bar n}$ is symmetric and positive semidefinite, so every player's cost depends on the full joint vector. Each instance is constructed to be merely monotone,
i.e., with zero minimum eigenvalue of $(F+F^\T)/2$ and $F$ invertible.
We consider $m$ shared inequality constraints and $q$ shared equality constraints, where the matrices $A$ and $E$ are generated as random sparse matrices with density 0.1. The right-hand sides of the constraints and linear terms of the costs are generated so that the problem is feasible and has $\lceil N/4\rceil$ active inequality constraints at the solution. 

\begin{table}[t]
  \setlength{\tabcolsep}{7pt}
  \renewcommand{\arraystretch}{1.} % Adjust row separation
  \centering
{\scriptsize
 \begin{tabular}{rrr|r|r|r|r|r|r|r}
  \hline
  $N$ & $m$ & $q$ & \multicolumn{1}{c|}{EG} & \multicolumn{1}{c|}{EG-aug} & \multicolumn{1}{c|}{APG-QP} & \multicolumn{1}{c|}{C} & \multicolumn{1}{c|}{C-r} & \multicolumn{1}{c|}{D-p} & \multicolumn{1}{c}{D-r-p} \\\hline
  %  &  &  & ms & ms & ms & ms & ms & ms & ms \\
  \hline
  2 & 1 & 0 & 7.08 & 0.04 & 0.07 & 0.09 & 0.06 & 0.04 & \textbf{0.02} \\
  2 & 4 & 0 & 4.92 & 0.16 & 0.13 & 0.10 & 0.07 & 0.04 & \textbf{0.02} \\
  5 & 2 & 0 & 3.31 & \textbf{0.01} & 0.09 & 0.14 & 0.06 & 0.04 & 0.02 \\
  5 & 10 & 0 & 3.08 & 0.56 & 0.27 & 0.21 & 0.11 & 0.05 & \textbf{0.02} \\
  5 & 10 & 1 & 2.56 & 1.18 & 0.31 & 0.20 & 0.11 & 0.05 & \textbf{0.02} \\
  10 & 5 & 0 & 2.97 & 0.65 & 0.21 & 0.35 & 0.07 & 0.05 & \textbf{0.02} \\
  10 & 20 & 0 & 3.03 & 2.44 & 0.66 & 0.59 & 0.27 & 0.08 & \textbf{0.04} \\
  10 & 20 & 3 & 2.73 & 5.84 & 0.95 & 0.59 & 0.27 & 0.09 & \textbf{0.04} \\
  20 & 10 & 0 & 3.12 & 4.78 & 0.67 & 1.38 & 0.12 & 0.09 & \textbf{0.03} \\
  20 & 40 & 0 & 5.05 & 18.83 & 3.14 & 2.75 & 1.02 & 0.29 & \textbf{0.09} \\
  20 & 40 & 6 & 5.01 & 38.78 & 4.86 & 2.70 & 1.02 & 0.35 & \textbf{0.10} \\
  50 & 25 & 0 & 15.96 & 34.37 & 11.66 & 16.16 & 0.46 & 0.77 & \textbf{0.08} \\
  50 & 100 & 0 & 51.26 & 92.22 & 80.68 & 36.93 & 10.29 & 4.36 & \textbf{1.02} \\
  50 & 100 & 16 & 86.93 & 222.50 & 196.67 & 82.77 & 12.08 & 5.49 & \textbf{1.68} \\
  100 & 50 & 0 & 71.78 & 89.61 & 135.20 & 112.28 & 1.92 & 6.09 & \textbf{0.87} \\
  100 & 200 & 0 & 220.23 & 295.05 & 843.63 & 197.68 & 72.78 & 33.63 & \textbf{6.38} \\
  100 & 200 & 33 & 393.50 & 665.48 & 1830.71 & 346.21 & 132.35 & 45.34 & \textbf{10.41} \\
  \hline
  \end{tabular}
}
\caption{
Mean CPU time (ms)
over 100 random instances. EG = extragradient method with projection via QP, EG-aug = outer-approximated extragradient method, 
C=Clarabel, D = DAQP, r = reduced QP~\eqref{eq:QP_dual}, p = proximal point iterations, APG-QP = accelerated proximal gradient. 
For EG, the gap is computed as in~\eqref{eq:Gap}, while for the other methods, it is $\lambda^\T(b-Ax)$. 
The same tolerance is used in EG-aug to ensure feasibility of $x$ when the algorithm stops. 
}
\label{tab:random-lqgne}
\end{table}

Table~\ref{tab:random-lqgne} reports the mean CPU time and average number of iterations over random instances of v-GNE LQ games with up to $N=100$ agents, $\bar n=2$ variables per agent ($n = N\bar n$), and different numbers $m$ of coupling inequality and $q$ of equality constraints. We collect 100 instances for each configuration as follows. We set tentative agents' QP Hessian matrices
$\tilde Q^{(i)} = B_i^\T B_i$, where each entry of $B_i$ is sampled i.i.d.\ from the standard normal distribution $\mathcal{N}(0,1)$, and compute the resulting pseudogradient matrix $F$. If $\lambda_{\rm min}(F)\leq 0$, we set $Q^{(i)}=\tilde Q^{(i)} -\lambda_{\rm min}(F)I$ to make the game merely monotone and add the instance, otherwise discard and repeat.
Each entry of the linear terms $c^{(i)}$ is sampled from $\mathcal{N}(0,5^2)$, 
and the coupling constraints are generated as follows: each entry of $A$ and $E$ is sampled from $\mathcal{N}(0,1)$, a feasible vector $x_0$ is sampled uniformly in $[-5,5]^n$, then we set $f=Ex_0$ and $b=Ax_0+\tilde b$, where each entry of $\tilde b$ is sampled uniformly in $[0.1, 0.5]$.

We consider the extragradient method (\texttt{EG})~\eqref{eq:extragradient} for solving merely-monotone variational inequalities as a baseline for comparison, with projection onto the constraint set $\XX$ performed by solving a 
strongly convex QP with DAQP, exploiting warm-starting of the active set from the previous projection
for numerical efficiency. We use the gap function~\eqref{eq:Gap}, computed by solving an additional linear program, as a stopping criterion.
Moreover, we consider the following solution methods for the QP reformulation
of the game: Clarabel for solving the QP~\eqref{eq:QP} (\texttt{C}) or the reduced QP~\eqref{eq:QP_dual}
(\texttt{C-r}), 
proximal point iterations and DAQP to solve \eqref{eq:QP} (\texttt{D-p}) 
or~\eqref{eq:QP_dual} (\texttt{D-r-p}), and the accelerated proximal gradient method 
with restart applied to the original QP~\eqref{eq:QP} (\texttt{APG-QP}) described in Section~\ref{sec:APG}.

The results show that \texttt{D-r-p} outperforms all other methods by a significant margin, especially as the size
of the merely monotone v-GNEP increases.

\subsection{Game-Theoretic MPC}

We consider an example of game-theoretic linear MPC with $N=2$ agents operating on the
discrete-time LTI plant~\eqref{eq:mpc_dyn}, where $x(t)\in\rr^{4}$ and $u(t),y(t)\in\rr^{4}$. Each agent $i$ controls two components of the input vector, denoted as $u_i(t)\in\rr^2$, and aims to regulate the output $y$ to a common set-point $r$, each using different output weight matrices, while satisfying local input constraints and shared output constraints.

The matrices ${\mathcal A}$, ${\mathcal B}$, ${\mathcal C}$ are generated randomly in $\NN(0,1)$ and, after generating
the matrices, the spectral radius of ${\mathcal A}$ is scaled to~$0.95$ and ${\mathcal C}$ is normalized so that the 
DC gain equals the identity. We use diagonal weights $Q_y^{(i)}$ with diagonal entries equal to $1$ except for those corresponding to $y_i$,  which are set to $2$, $Q_{\Delta u}^{(i)}=0.05I$,
$Q_{\epsilon_1}=Q_{\epsilon_2}=10^{3}$; 
each entry of $u_{\min}=-3$, $u_{\max}=3$, $\Delta u_{\min}=-\infty$, $\Delta u_{\max}=\infty$,
and we consider both the presence and the absence of output constraints
$y_{\min}=0$, $y_{\max}=2$ (output constraints become shared inequality constraints in~\eqref{eq:QPi}).

We adopt the linear GT-MPC setup described in Section~\ref{sec:mpc}
for different prediction horizons $T\in\{5, 10, 15, 20, 30\}$. The resulting game is verified numerically to be strongly monotone for each value of~$T$. For this reason, we also compare the performance of the methods described in Section~\ref{sec:solution_methods} with the solver proposed in~\cite{ABB26} 
for strongly monotone LQ-GNEPs (\texttt{dr\_daqp}), that we apply either directly on
the resulting GNE or on its condensed version obtained by eliminating the state variables
via SVD decomposition of the equality constraint matrix $E$. The proximal point iterations D-p and D-r-p are applied 
with $\frac{1}{\gamma_k}\equiv \rho=10^{-6}$.

\begin{figure}[ht]
\centering
\begin{subfigure}{.8\columnwidth}
\centering
\includegraphics[width=.8\columnwidth]{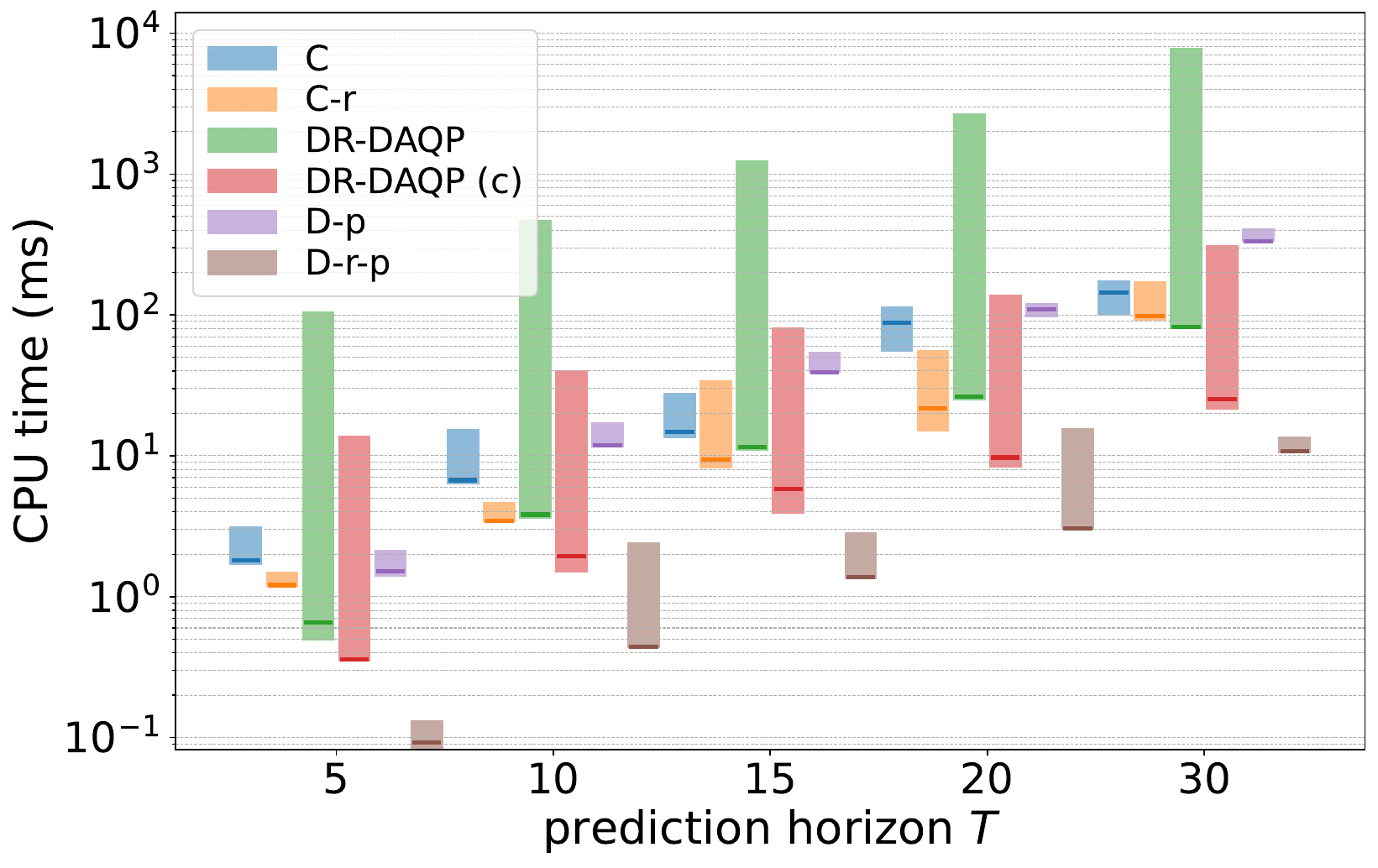}
\caption{Input constraints only.}
\label{fig:mpc_timing_a}
\end{subfigure}
\begin{subfigure}{.8\columnwidth}
\centering
\includegraphics[width=.8\columnwidth]{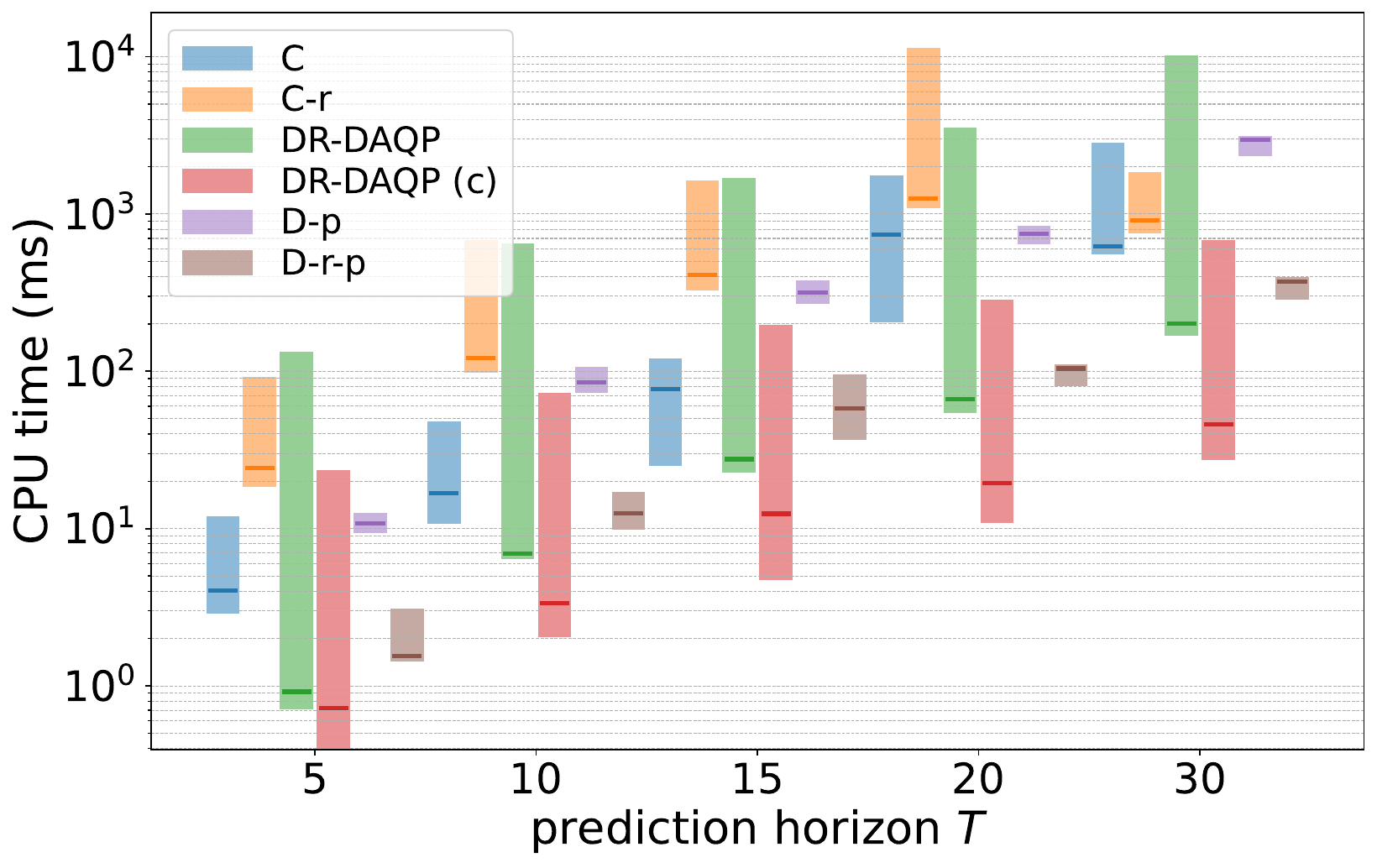}
\caption{Input and output constraints.}
\label{fig:mpc_timing_b}
\end{subfigure}
\caption{CPU time per closed-loop step for GT-MPC ($N=2$ agents).
Bars show the min/max range, horizontal marks indicate the median.}
\label{fig:mpc_timing}
\end{figure}

Figure~\ref{fig:mpc_timing} reports the minimum, median, and maximum CPU time per
closed-loop step over the simulation.
The results show that for input-constrained problems, D-r-p outperforms all the other methods. In the presence of output constraints, DR-DAQP has a better median CPU time, although D-r-p has a better worst-case one, which is more relevant for real-time applications.

\section{Conclusions}
We showed that computing a v-GNE of an LQ-GNEP with shared affine constraints, under mere monotonicity, reduces to solving a single convex QP from the players' joint KKT conditions, with zero optimal value iff a v-GNE exists. From this, we derived a regularized variant with an explicit suboptimality bound and two accelerated methods (proximal-point, projected-gradient) with $\mathcal{O}(1/k^2)$ last-iterate guarantees, versus $\mathcal{O}(1/\sqrt k)$ for extragradient, plus a reduced dual-only QP for invertible pseudogradients. Numerical results, including a game-theoretic MPC case study, confirm substantial gains over existing solvers as problem size grows. 
Future work includes extending to non-monotone LQ-GNEPs and studying inexact QP subproblem solutions for real-time MPC.

\appendix

\section{Appendix}

\subsection{Proof of Lemma~\ref{lem:QP_dual}}
\label{app:QP_dual}
Since $F$ is invertible, from~\eqref{eq:QP_stat} we get
$x= -G(f + A^\T \lambda + E^\T \nu)$. By substituting this expression in~\eqref{eq:QP_eq}
we get $-EG(f + A^\T \lambda + E^\T \nu)=e$, i.e., 
% -EF^{-1}E^\T\nu - EF^{-1}(f + A^\T \lambda) - e = 0$, i.e., 
%  EF^{-1}E^\T\nu + EF^{-1}(f + A^\T \lambda) + e = 0$, i.e., 
%  M\nu = -EF^{-1}(f + A^\T \lambda) - e = 0$, i.e., 
%  \nu = -M^{-1}EF^{-1}(f + A^\T \lambda) - M^{-1}e = 0$, i.e., 
$\nu = -M^{-1}(EG(f + A^\T \lambda) + e)$.  %[OK!]
This gives
$x = -G(f + A^\T \lambda - E^\T M^{-1}(EG(f + A^\T \lambda) + e))
    % Constant term: 
        % -G(f - E^\T M^{-1}(EGf + e)) \\
        % = -G((I - E^\T M^{-1}EG) f  - E^\T M^{-1}e) \\
        % = -P f + F^{-1}E^\T M^{-1}e = x_0 [OK!]\\
    % Linear term: 
        % -F^{-1}(A^\T \lambda - E^\T M^{-1}(EF^{-1}(A^\T \lambda))) \\
        % -F^{-1}A^\T \lambda +F^{-1} E^\T M^{-1}(EF^{-1}(A^\T \lambda))) \\
        % -F^{-1}A^\T \lambda +F^{-1} E^\T M^{-1}EF^{-1}A^\T \lambda))) \\
        % (-F^{-1} +F^{-1} E^\T M^{-1}EF^{-1})A^\T \lambda)) \\
        % -(F^{-1} -F^{-1} E^\T M^{-1}EF^{-1})A^\T \lambda)) [OK!]
 = x_0 - PA^\T\lambda$. Then, $b-Ax = b-A(x_0 - PA^\T\lambda) = q_0 + APA^\T\lambda = q_0 + H\lambda$.
From the proof of Lemma~\ref{lem:QP_reformulation}, we recall 
from~\eqref{eq:slack-minimization} that the objective of~\eqref{eq:QP} is equal to the complimentarity
gap
\[
    \lambda^\T (b - A x) = % \lambda^\T (b - A x_0 +APA^\T\lambda) = 
     \lambda^\T(q_0 + H\lambda) = \lambda^\T H_s\lambda + q_0^\T\lambda.
\]
Let us now prove that the symmetric part of $H$ is positive semidefinite. Consider a generic vector $x\in\rr^n$ and let $y = F^{-1}x$. Then,
\[
     \begin{aligned}
     x^\T(F^{-1} + F^{-\T})x &= 
     x^\T F^{-\T}F^{\T}F^{-1}x+x^\T F^{-\T}FF^{-1}x \\
    &=\left( y^\T F^\T y + y^\T F y \right) = 2y^\T F_s y \geq 0
     \end{aligned}
\]
for all $x \in \rr^n$, i.e., $G=F^{-1}$ has a positive semidefinite symmetric part. 

When $q=0$, we get $P=G=F^{-1}$, $H=AF^{-1}A^\T$, and, hence, 
$H_s = (H+H^\T)/2 = A(F^{-1} + F^{-\T})A^\T/2$ is positive semidefinite.

Consider now the case $q>0$ and let $R=E^\T M^{-1}E$ and $Q= I-RG$. 
Then 
\begin{equation}
    \begin{aligned}
    Q^\T GQ = &(I-G^\T R^\T)G(I-RG) \\
         = & G-GRG-G^\T R^\T G + G^\T R^\T GRG \\
         = & P -G^\T R^\T G + G^\T R^\T GRG 
    \end{aligned}
\label{eq:QGQ}
\end{equation}
The last term in~\eqref{eq:QGQ} can be rewritten as
\[
    \begin{aligned}
    G^\T R^\T GRG &= G^\T E^\T M^{-\T} (E G  E^\T) M^{-1}E G \\
                  &= G^\T E^\T M^{-\T} M M^{-1}E G \\
                  &= G^\T E^\T M^{-1}E G = G^\T R^\T G
    \end{aligned}
\]
and, therefore, the last two terms in~\eqref{eq:QGQ} cancel out, and we get $Q^\T GQ = P$. 
Hence, $H_s=A(P+P^\T)A^\T/2=(AQ^\T) (F^{-1}+F^{-\T})(QA^\T)/2$ is also positive semidefinite.
\cvd

\end{document}